\documentclass[twoside, 12pt]{article}
\usepackage{mathrsfs}
\usepackage{amssymb, amsmath, mathrsfs, amsthm}%, wasysym}
\usepackage{graphicx}
\usepackage{color}
\usepackage{tikz}
\usepackage[top=2cm, bottom=2cm, left=2cm, right=2cm]{geometry}
\usepackage{float, caption, subcaption}
\usepackage{diagbox}
\usepackage{algorithm}
\usepackage{booktabs}
\usepackage{algpseudocode}
\usepackage{amsmath}
\usepackage[colorlinks,
  linkcolor=blue,
 anchorcolor=blue,
 citecolor=blue]{hyperref} % 棰滆壊鍙皟
\input{epsf}

\newcommand{\bd}{\begin{description}}
\newcommand{\ed}{\end{description}}
\newcommand{\bi}{\begin{itemize}}
\newcommand{\ei}{\end{itemize}}
\newcommand{\be}{\begin{enumerate}}
\newcommand{\ee}{\end{enumerate}}
\newcommand{\beq}{\begin{equation}}
\newcommand{\eeq}{\end{equation}}
\newcommand{\beqs}{\begin{eqnarray*}}
\newcommand{\eeqs}{\end{eqnarray*}}

\definecolor{DarkGreen}{rgb}{0.2, 0.6, 0.3}

\usepackage{multicol}
\usepackage{amsthm}
\usepackage{multirow}
\usepackage{makecell}

\newtheorem{theorem}{Theorem}
\newtheorem{conjecture}{Conjecture}

\newtheorem{definition}{Definition}[section]
\newtheorem{corollary}[theorem]{Corollary}
\newtheorem{case}{Case}

\newtheorem{proposition}{Proposition}[section]

\begin{document}
\title{\textbf{On two conjectures for generalized off-diagonal Schur numbers}
\footnote{Supported by the National Science Foundation of China
(Nos. 12471329 and 12061059).}
}
\author{Yanyan Song\footnote{School of Mathematics and Statistis, Qinghai Normal University, Xining, Qinghai 810008, China. {\tt songyanyan@hrbeu.edu.cn}}, \ \ Yaping Mao\footnote{Academy of Plateau
Science and Sustainability, and School of Mathematics and Statistics, Qinghai Normal University, Xining, Qinghai 810008, China. {\tt yapingmao@outlook.com; myp@qhnu.edu.cn}}
}

\date{}
\maketitle
\begin{abstract}
For an integer $t \geq 3$, let $\mathcal{L}(t)$ denote the linear equation
$x_1 + x_2 + \cdots + x_{t-1} = x_t,$
where all variables are positive integers. For integers $k \geq 1$ and $t_0,t_1,\dots,t_{k-1} \geq 3$, the generalized Schur number $S(k;t_0,t_1,\dots,t_{k-1})$ is the least positive integer $N$ such that every $k$-coloring of $[1,N]$, for some $i \in \{0,1,\dots,k-1\}$, a solution to $\mathcal{L}(t_i)$ with all variables monochromatic in color $i$.
In 2015,
Ahmed and Schaal proposed a conjecture:
$S(3 ; 3, t, u)>3 t u-t u-u-1$ for $3=t<u$ and $3<t \leq u$. In this paper, we confirm this conjecture. At the same paper, they also conjecture that
$S(3 ; s, t, u)=s t u-t u-u-1$ for $4 \leq s \leq t \leq u$. Motivated by the second conjecture, we give a recursive lower bound of
$S(r; k_0, k_1, \dots, k_{r-1})$ and upper bounds for $S(r; k_0, k_1, \dots, k_{r-1})$ and $S(r;k_0,\dots,k_{r-2},u)$ for all sufficiently large $u$.
\\[2mm]
{\bf Keywords:} Schur number; 2-coloring; Monochromatic solution; Ramsey theory; Upper bound estimation\\[2mm]
{\bf AMS subject classification 2020:} 05D10; 05C15, 05A17
\end{abstract}

\section{Introduction}

Let $\mathbb{N}$, $\mathbb{Z}$, and $\mathbb{Q}$ denote the sets of natural numbers, integers, and rational numbers, respectively. For integers $a \leq b$, we write
\[
[a,b] = \{n \in \mathbb{Z} \mid a \leq n \leq b\}.
\]
Throughout, variables such as $x,y,z,w$ (with or without subscripts) represent integer variables, while unspecified positive integer coefficients are denoted by $a,b,c,d$ (again, possibly with subscripts). We assume all coefficients are strictly positive; consequently, the equation $ax+by+cz=dw$ is considered distinct from
$ax+by=cz+dw.$

A function $\Delta: [1,n]\longrightarrow \{1,2,\ldots,r\}$ is called an \emph{$r$-coloring} of $[1,n]$, where the integers $1,2,\ldots,r$ represent $r$ distinct colors. Given an $r$-coloring $\Delta$ and a system $\mathcal{L}$ of linear equations or inequalities in $k$ variables, a solution $(x_1,x_2,\ldots,x_k)$ of $\mathcal{L}$ is called \emph{monochromatic} if
\[
\Delta(x_1)=\Delta(x_2)=\cdots=\Delta(x_k).
\]

In 1916, Schur \cite{Schur16} established the following fundamental result, which is regarded as one of the earliest cornerstones of Ramsey theory, alongside Hilbert’s “cube lemma” (see, e.g., \cite{Soifer00}).

\begin{theorem}{\upshape \cite{Schur16}}\label{th-Schur16-1}
{\upshape (Schur's Theorem, finite version)}
For every $r \in \mathbb{Z}^+$, there exists a positive integer $S(r)$ such that any $r$-coloring
\[
\Delta:[1,S(r)] \longrightarrow [r]
\]
contains a monochromatic triple $(x,y,z)$ satisfying $x+y=z.$
\end{theorem}

\begin{definition}[Schur Number]
For a given positive integer $n$, the \emph{Schur number} $S_2(r)$ is defined as the smallest positive integer $N$ such that any $r$-coloring of $[1,N]$ such that there ia a monochromatic solution to equation $x+y=z$,where $x \leq y$.
\end{definition}

The exact value of $S_2(4)$ was computed by Baumert \cite{Ba61}, and more recently Heule \cite{He18} determined $S_2(5)$. In addition, Fredricksen and Sweet \cite{FS00} obtained lower bounds for $S_2(6)$ and $S_2(7)$ using symmetric sum-free partitions. The currently best-known values and bounds for $S_2(r)$, $1 \leq r \leq 7$, are listed in Table~\ref{table-1}.  For all integers \( r \geq 1 \), For all integers \( r \geq 1 \), Schur  \cite{Schur16} established the following general lower and upper bounds for \( S_2(r) \), that is,
$\frac{3^r + 1}{2} \leq S_2(r) \leq \left\lfloor r! e \right\rfloor.$
In addition, he derived the recursive inequality $S_2(r) \geq 3S_2(r - 1) - 1.$

\begin{table}[h]
\centering
\caption{The first few Schur numbers $S_2(r)$.}
\begin{tabular}{c ccccccc}
\toprule
$n$ & 1 & 2 & 3 & 4 & 5 & 6 & 7 \\
\midrule
$S_2(r)$ & 2 & 5 & 14 & 45 & 161 & $\ge 537$ & $\ge 1681$ \\
\bottomrule
\end{tabular}
\label{table-1}
\end{table}

\begin{definition}[Generalized Schur Number]
For an integer $t \geq 3$, let $\mathcal{L}(t)$ denote the linear equation
$$x_1 + x_2 + \cdots + x_{t-1} = x_t,$$
where all variables are positive integers. For integers $r \geq 1$ and $k_0,k_1,\dots,k_{r-1} \geq 3$, the \emph{generalized Schur number} $S(r; k_0, k_1, \dots, k_{r-1})$ is the least positive integer $N$ such that every $k$-coloring of $[1,N]$, for some $i \in \{0,1,\dots,r-1\}$, a solution to $\mathcal{L}(t_i)$ with all variables monochromatic in color $i$.
\end{definition}

The \emph{generalized diagonal Schur number} is defined by
$S(r;k) = S(r; k_0,\dots,k_{r-1})$ with $k_i=k$ for all $0\leq i \leq r-1$.
For distinct $k_i$, $S(r; k_0,\dots,k_{r-1})$ is called
the \emph{generalized off-diagonal Schur number}.

In 1966, Znám \cite{Zn66} derived the lower bound
$S(r;k) \geq \frac{k-1}{k}\left((k+1)^r-1\right)+1.$
Irving \cite{Irwing73} improved the upper bound for $S(r,k)$ in 1973 to
$S_k(n) \leq \left\lfloor \sqrt[k-1]{e}\, n!\, (k-1)^n \right\rfloor.$ Later, in 1982, Beutelspacher and Brestovansky \cite{BB82} proved that
$S(2;k)=k^2-k-1$ for $k \geq 2.$
Moreover, in 2016, Ahmed and Schaal \cite{AhSc16} computed the exact values of $S(3;k)$ for $k=4,5,6$. Later, in 2019, Boza et al. \cite{BMRS19} established that
$S(3;k)=k^3-k^2-k-1$ for $k \geq 3,$ by proving an upper bound that coincides with the lower bound previously obtained by Zn\'{a}m \cite{Zn66}.
The following table lists some exact value generalized Schur numbers $S(3;s,t,u)$.  For different variations, results, and references, see \cite{AhmedEtAl13,ABRS23,BialostockiSchaal00,BiSc00, BRS18, Budden20,EMRS12, EFFMRT13, Exoo94, KezdyEtAl09,LandmanRobertson04,MartinelliSchaal07,Schaal93,SchaalSnevily08,Schur16}.

\begin{table}[!htbp]
\footnotesize
\caption{Some exact values of $S(3;s,t,u)$.}
\centering
\begin{tabular}{lc}
\hline
$S(3;s,t,u)$ & References \\
\hline
$S(3;3,3,3)=14$, & \cite{Schur16} \\
\makecell[l]{
$S(3;3,3,4)=23,\ S(3;3,3,5)=32,\ S(3;3,3,6)=41,\ S(3;3,3,7)=49,$\\
$S(3;3,4,4)=31,\ S(3;3,4,5)=47,\ S(3;3,4,6)=49,\ S(3;3,4,7)=59,$\\
$S(3;3,5,5)=58,\ S(3;3,5,6)=70,\ S(3;3,5,7)=80,$\\
$S(3;3,6,6)=85,\ S(3;3,6,7)=107,$\\
$S(3;4,4,4)=43,\ S(3;4,4,5)=54,\ S(3;4,4,6)=65,\ S(3;4,4,7)=76,$\\
$S(3;4,5,5)=69,\ S(3;4,5,6)=83,\ S(3;4,5,7)=97,$\\
$S(3;4,6,6)=101,$\\
$S(3;5,5,5)=94,\ S(3;5,5,6)=113,$\\
$S(3;6,6,6)=173.$\\
} & \cite{AhSc16} \\
\hline
\end{tabular}
\label{table-2}
\end{table}

Robertson and Schaal \cite{RobertsonSchaal01} established the complete $2$-color result as follows.

\begin{theorem}{\upshape \cite{RobertsonSchaal01}}\label{RobertsonSchaal01}
For integers $s,t \geq 3$, we have
$$
S(2;s,t) =
\begin{cases}
3t-4 & \text{if } s=3 \text{ and } t \equiv 1 \pmod{2}, \\
3t-5 & \text{if } s=3 \text{ and } t \equiv 0 \pmod{2}, \\
st - t - 1 & \text{if } 4 \leq s \leq t.
\end{cases}
$$
\end{theorem}

In this work, we resolve the following two conjectures for 3-color off-diagonal generalized Schur numbers.

\begin{conjecture}{\upshape \cite{AhSc16}}\label{conjecture-1}
For $3=t<u$ and $3<t \leq u$, we have
$$
S(3 ; 3, t, u)>3 t u-t u-u-1 .
$$
\end{conjecture}

\begin{conjecture}{\upshape \cite{AhSc16}}\label{conjecture-2}
For $4 \leq s \leq t \leq u$, we have
$$
S(3 ; s, t, u)=s t u-t u-u-1.
$$
\end{conjecture}

For brevity, we restate the core object of study: $S(3;s,t,u)$ is the least $N$ such that every 3-coloring of $[1,N]$ contains a monochromatic $\mathcal{L}(s)$ in color $1$, a monochromatic $\mathcal{L}(t)$ in color $2$, or a monochromatic $\mathcal{L}(u)$ in color $3$.

\section{A recursive lower bound}

Ahmed and Schaal in \cite{AhSc16} obtained a general lower bound, as shown below.

\begin{theorem}{\upshape \cite{AhSc16}}\label{AhSc16}
For all $r \geq 1$ and for all $k_0, k_1, \dots, k_{r-1}$ where $3 \leq k_0 \leq k_1 \leq \dots \leq k_{r-1}$,
\[
S(r; k_0, k_1, \dots, k_{r-1}) \geq \prod_{j=0}^{r-1} k_j - \sum_{i=1}^{r-1} \prod_{j=i}^{r-1} k_j - 1.
\]
\end{theorem}

In the proof of \cite[Theorem 2.11]{AhSc16}, the following conclusion is established.
Let $r\ge 2$, let $3\le k_0\le k_1\le \cdots \le k_{r-1}$.
\begin{equation}\label{eq_lower}
S(r;k_0,\dots,k_{r-1})\ge k_{r-1}\,S(r-1;k_0,\dots,k_{r-2})-1
\end{equation}

\begin{theorem}\label{th:iterated-recursive-lower}
Let $r\ge 3$ and let $3\le k_0\le k_1\le \cdots \le k_{r-1}$.  Then for every integer $m$ with
$2\le m\le r-1$,
\[
S(r;k_0,\dots,k_{r-1})
\ge
\Bigl(\prod_{j=m}^{r-1}k_j\Bigr)\,S(m;k_0,\dots,k_{m-1})
-\sum_{i=m}^{r-1}\prod_{j=i+1}^{r-1}k_j.
\]
Here an empty product is interpreted as $1$.
\end{theorem}

\begin{proof}
We induct on $r-m$.  If $r-m=1$, then the formula is exactly \eqref{eq_lower}.
Assume now that $r-m\ge 2$ and that the statement holds with $r$ replaced by $r-1$.
By \eqref{eq_lower},
\[
S(r;k_0,\dots,k_{r-1})\ge k_{r-1}\,S(r-1;k_0,\dots,k_{r-2})-1.
\]
Applying the induction hypothesis to the factor $S(r-1;k_0,\dots,k_{r-2})$ gives
\[
S(r-1;k_0,\dots,k_{r-2})
\ge
\Bigl(\prod_{j=m}^{r-2}k_j\Bigr)S(m;k_0,\dots,k_{m-1})
-\sum_{i=m}^{r-2}\prod_{j=i+1}^{r-2}k_j.
\]
Substituting this into the previous inequality yields
\[
\begin{aligned}
S(r;k_0,\dots,k_{r-1})
&\ge k_{r-1}\left[
\Bigl(\prod_{j=m}^{r-2}k_j\Bigr)S(m;k_0,\dots,k_{m-1})
-\sum_{i=m}^{r-2}\prod_{j=i+1}^{r-2}k_j
\right]-1\\
&=
\Bigl(\prod_{j=m}^{r-1}k_j\Bigr)S(m;k_0,\dots,k_{m-1})
-\sum_{i=m}^{r-2}\prod_{j=i+1}^{r-1}k_j-1\\
&=
\Bigl(\prod_{j=m}^{r-1}k_j\Bigr)S(m;k_0,\dots,k_{m-1})
-\sum_{i=m}^{r-1}\prod_{j=i+1}^{r-1}k_j,
\end{aligned}
\]
as required.
\end{proof}

\begin{corollary}\label{cor:iterated-from-two-color}
For every $r\ge 3$ and every $3\le k_0\le \cdots \le k_{r-1}$,
\[
S(r;k_0,\dots,k_{r-1})
\ge
\Bigl(\prod_{j=2}^{r-1}k_j\Bigr)S(2;k_0,k_1)
-\sum_{i=2}^{r-1}\prod_{j=i+1}^{r-1}k_j.
\]
In particular, if $4\le s\le t\le k_2\le \cdots \le k_{r-1}$, then
\[
S(r;s,t,k_2,\dots,k_{r-1})
\ge
\Bigl(\prod_{j=2}^{r-1}k_j\Bigr)(st-t-1)
-\sum_{i=2}^{r-1}\prod_{j=i+1}^{r-1}k_j.
\]
\end{corollary}

\begin{proof}
Upon setting $m=2$ in Theorem~\ref{th:iterated-recursive-lower}, the second displayed equation is a consequence of \cite{RobertsonSchaal01}, namely
$S(2;s,t)=st-t-1$ for $4\le s\le t.$
\end{proof}

Using Theorem \ref{th:iterated-recursive-lower}, we are able to give an alternative proof of Theorem \ref{AhSc16}.

\begin{proof}[Proof of Theorem \ref{AhSc16}]
We proceed by induction on $r$. For $r=1$, it is readily seen that $S(1;k_0)=k_0-1$. Suppose the statement holds for $r-1$, we now establish the result for $r$. For $2 \leq m \leq r-1$,
by the induction hypothesis, we have
\[
S(m; k_0, \dots, k_{m-1}) \geq \prod_{j=0}^{m-1} k_j - \sum_{i=1}^{m-1} \prod_{j=i}^{m-1} k_j - 1,
\]
By Theorem $\ref{th:iterated-recursive-lower}$, we have
\[
\begin{aligned}
S(r; k_0, \dots, k_{r-1}) &\geq \left( \prod_{j=m}^{r-1} k_j \right) \left( \prod_{j=0}^{m-1} k_j - \sum_{i=1}^{m-1} \prod_{j=i}^{m-1} k_j - 1 \right) - \sum_{i=m}^{r-1} \prod_{j=i+1}^{r-1} k_j \\
&= \prod_{j=0}^{r-1} k_j - \sum_{i=1}^{m-1} \prod_{j=i}^{r-1} k_j - \prod_{j=m}^{r-1} k_j - \sum_{i=m}^{r-1} \prod_{j=i+1}^{r-1} k_j.
\end{aligned}
\]
Note that
\[
\prod_{j=m}^{r-1} k_j + \sum_{i=m}^{r-1} \prod_{j=i+1}^{r-1} k_j = \sum_{i=m}^{r-1} \prod_{j=i}^{r-1} k_j + 1,
\]
since the empty product appearing in the final term is equal to $1$.
Consequently, the aforementioned inequality reduces to
\[
S(r; k_0, \dots, k_{r-1}) \geq \prod_{j=0}^{r-1} k_j - \sum_{i=1}^{r-1} \prod_{j=i}^{r-1} k_j - 1,
\]
which is precisely the stated bound.
\end{proof}

\section{Complete solution to Conjecture \ref{conjecture-1}}

In what follows, we present lower bounds for the two conjectures.

\begin{proposition}\label{pro-1}
Let $3 \leq s \leq t \leq u$ be integers.
\begin{itemize}
\item If $4 \leq s \leq t \leq u$, then
$
S(3;s,t,u) \geq stu-tu-u-1.
$
\item If $3=t<u$ and $3<t \leq u$, then
$
S(3;3,t,u) >2tu-u-1.
$
\end{itemize}
\end{proposition}

\begin{proof}
If $4 \leq s \leq t \leq u$, then it follows from Corollary \ref{cor:iterated-from-two-color} that
$S(3;s,t,u) \geq u (st-t-1) -1=stu-tu-u-1.$

Let $N=2tu-u-1$ and $\Delta$ be a $3$-coloring defined as
$\Delta : \left[1, N\right] \to \{1,2,3\},$
where $1,2,3$ denote three distinct colors. For $i = 1,2,3$, let $A_i = \Delta^{-1}(i)$. It is straightforward to verify that $\left[1, N\right] = A_1 \cup A_2  \cup A_3.$ We shall show that there exists a coloring of $[1,N]$ admitting no monochromatic copy of $\mathcal{L}(3)$ in color $1$, no monochromatic copy of $\mathcal{L}(t)$ in color $2$, and no monochromatic copy of $\mathcal{L}(u)$ in color $3$.

We next consider the cases $3 = t < u$ and $3 < t \leq u$. In what follows, we split into the following three cases. It was shown in \cite{AhSc16} that \(S(3;3,3,4)=23>19\). Therefore, in what follows, we restrict our attention to the case \(u \geq 5\) for \(S(3;3,3,u)\).

\begin{case}
$t=3$ and $u \geq 5$.
\end{case}

Then $N=5u-1$. Let
\begin{equation*}
\begin{cases}
A_1 =\{1,4,6,9,5u-5,5u-2\}, \\
A_2 =\{2,3,7,8,5u-4,5u-3\}, \\
A_3 =\{5,5u-1\} \cup [10,5u-6].
\end{cases}
\end{equation*}

Let \(x_1 \leq x_2 \) and $x_3=x_1+x_2$ be integers colored \(1\). We first verify that the triple \((x_1,x_2,x_3)\) is not a solution to \(\mathcal{L}(3)\).
If $x_1=1$ and $x_2 \in \{1,4,6,9,5u-5,5u-2\}$, then $x_3 \in \{2,5,7,10,5u-4,5u-1\}
\nsubseteq A_1$, and hence $(x_1, x_2, x_3 )$ is not a solution to \(\mathcal{L}(3)\) in color $1$.
If $x_1=4$ and $x_2 \in \{4,6,9,5u-5,5u-2\}$, then $x_3 \in \{8,10,13, 5u-1,5u+2\} \nsubseteq A_1$, and hence $(x_1, x_2, x_3 )$ is not a solution to \(\mathcal{L}(3)\) in color $1$.
It follows from \(u \geq 5\) that \(5u-5 \geq 20\).
If $x_1, x_2 \in \{6,9\}$, then $x_3 \in \{12,15,18\} \nsubseteq A_1$, and hence $(x_1, x_2, x_3 )$ is not a solution to \(\mathcal{L}(3)\) in color $1$. If $x_1 \in \{6,9\}$ and $x_2 \in \{5u-5,5u-2\}$, then $x_3 \geq 5u+1 >N$, and hence $(x_1, x_2, x_3 )$ is not a solution to \(\mathcal{L}(3)\) in color $1$. If $x_1, x_2 \in \{5u-5,5u-2\}$, then $x_3 \geq 10u-10 > 5u-1=N$ for $u \geq 5$, and hence $(x_1, x_2, x_3 )$ is not a solution to \(\mathcal{L}(3)\) in color $1$.

Let \(x_1 \leq x_2\)  and $x_3=x_1+x_2$ be integers colored \(2\). We first verify that the triple \((x_1,x_2,x_3)\) is not a solution to \(\mathcal{L}(3)\).
If $x_1 \in \{2,3\}$ and $x_2 \in \{2,3,7,8,5u-4,5u-3\}$, then $x_3\in \{4,5,6,9,10,11,5u-2,5u-1,5u\} \nsubseteq A_2$, and hence $(x_1, x_2, x_3 )$ is not a solution to \(\mathcal{L}(3)\) in color $2$.
It follows from \(u \geq 5\) that \(5u-4 \geq 21\).
If $x_1, x_2 \in \{7,8\}$, then $x_3 \in \{14,15,16\} \nsubseteq A_1$, and hence $(x_1, x_2, x_3 )$ is not a solution to \(\mathcal{L}(3)\) in color $2$.
If $x_1 \in \{7,8\}$ and $x_2 \in \{5u-4,5u-3\}$, then $x_3 \geq 5u+3 >N$, and hence $(x_1, x_2, x_3 )$ is not a solution to \(\mathcal{L}(3)\) in color $2$. If $x_1, x_2 \in \{5u-4,5u-3\}$, then $x_3 \geq 10u-8> 5u-1=N$ for $u \geq 5$, and hence $(x_1, x_2, x_3 )$ is not a solution to \(\mathcal{L}(3)\) in color $2$.

Let \(x_1 \leq x_2 \leq \ldots \leq x_{u-1}\) and $x_u=\sum_{i=1}^{u-1} x_i$ be integers colored \(3\). We first verify that \((x_1,x_2,\ldots, x_u)\) is not a solution to \(\mathcal{L}(u)\).
If any $x_i=5$ for $1 \leq i \leq u-1$,
then $x_u=5u-5 \notin A_3$, and hence \((x_1,x_2,\ldots, x_u)\) is not a solution to \(\mathcal{L}(u)\) in color $3$.
If $x_{u-1} \geq 10$, then $x_u \geq 5(u-2)+10=5u >N$, and hence \((x_1,x_2,\ldots, x_u)\) is not a solution to \(\mathcal{L}(u)\) in color $3$.
It follows that $S(3;3,3,u)>N=5u-1$.

\begin{case}
$t=4$ and $u \geq 4$.
\end{case}

Then $N=7u-1$. Let
\begin{equation*}
\begin{cases}
A_1 =\{1,6,8,7u-7,7u-5\}, \\
A_2= [2,5]\cup [7u-4, 7u-1], \\
A_3 =\{7,7u-6\} \cup [9,7u-8].
\end{cases}
\end{equation*}

Let \(x_1 \leq x_2 \leq x_3\) and $x_3=x_1+x_2$ be integers colored \(1\). We first show that the triple \((x_1,x_2,x_3)\) cannot be a monochromatic solution to \(\mathcal{L}(3)\) in color \(1\).
If \(x_1=1\) and \(x_2 \in \{1,6,8,7u-7,7u-5\}\), then $x_3 \in \{2,7,9,7u-6,7u-4\} \nsubseteqq A_1.$
Hence, \((x_1,x_2,x_3)\) is not a solution to \(\mathcal{L}(3)\) in color \(1\). Since \(u \geq 4\), we have \(7u-7 \geq 21\). If \(x_1,x_2 \in \{6,8\}\), then
$x_3 \in \{12,14,16\} \nsubseteqq A_1,$
and so \((x_1,x_2,x_3)\) is not a solution to \(\mathcal{L}(3)\) in color \(1\). If \(x_1 \in \{6,8\}\) and \(x_2 \in \{7u-7,7u-5\}\), then $x_3=7u-1 \notin A_1$ or $x_3 \geq 7u+1 > N,$
and therefore \((x_1,x_2,x_3)\) is not a solution to \(\mathcal{L}(3)\) in color \(1\). Finally, if \(x_1,x_2 \in \{7u-7,7u-5\}\), then $x_3 \geq 14 u-14 > 5u-1 =N $ for \(u \geq 4\). Hence, \((x_1,x_2,x_3)\) is not a solution to \(\mathcal{L}(3)\) in color \(1\).

Let \(x_1 \leq x_2 \leq x_3 \leq x_4\) and $x_4=\sum_{i=1}^3 x_i$ be integers colored \(2\). We first show that \((x_1,x_2,x_3,x_4)\) cannot be a monochromatic solution to \(\mathcal{L}(4)\) in color \(2\). If $x_3 \geq 7u-4$, then $x_4 \geq 4+7u-4=7u >N$.
Since $u \geq 4$, it follows that $7u-4 \geq 24$. If $x_3 \leq 5$, then $ 6 \leq x_4 \leq 15$. Thus, $x_4 \notin A_2$. Therefore, \((x_1,x_2,x_3,x_4)\) is not a monochromatic solution to \(\mathcal{L}(4)\) in color \(2\).

Let \(x_1 \leq x_2 \leq \ldots \leq x_{u-1}\) and $x_u=\sum_{i=1}^{u-1} x_i$ be integers colored \(3\). We first verify that \((x_1,x_2,\ldots, x_u)\) is not a solution to \(\mathcal{L}(u)\).
If any $x_i=7$ for $1 \leq i \leq u-1$,
then $x_u=7u-7 \notin A_3$, and hence \((x_1,x_2,\ldots, x_u)\) is not a solution to \(\mathcal{L}(u)\) in color $3$.
If $x_{u-1} \geq 9$, then $x_u \geq 7(u-2)+9=7u-5$, that is, $x_u \notin A_3$,  and hence \((x_1,x_2,\ldots, x_u)\) is not a solution to \(\mathcal{L}(u)\) in color $3$. It follows that $S(3;3,4,u)>N=7u-1$.

\begin{case}
$u \geq t \geq 5$.
\end{case}

For $u \geq t \geq 5$, from Theorem \ref{AhSc16} and Corollary \ref{cor:iterated-from-two-color}, we have
$$
S(3;s,t,u) \geq u \cdot S(2;s,t)-1
=
\begin{cases}
3tu-4u-1 & \text{if }t \equiv 1 \pmod{2}, \\
3tu-5u-1 & \text{if }t \equiv 0 \pmod{2}.
\end{cases}
$$
We consider two cases according to the parity of \(t\). If \(t\) is odd, that is, \(t \equiv 1 \pmod{2}\), then \(t \geq 5\), and hence $3tu-4u-1 > 2tu-u-1.$
If \(t\) is even, that is, \(t \equiv 0 \pmod{2}\), then \(t \geq 6\), and therefore, $3tu-5u-1 > 2tu-u-1.$ Then $S(3;s,t,u)>2tu-u-1.$
\end{proof}

\section{Upper bounds and Conjecture \ref{conjecture-2}}

We now derive the relationship between generalized Schur number and Ramsey number.

\begin{theorem}\label{th:ramsey-embedding}
For every $r\ge 2$ and every $k_0,\dots,k_{r-1}\ge 3$, we have
\[
S(r;k_0,\dots,k_{r-1})\le R_r(k_0,\dots,k_{r-1})-1.
\]
In particular, $S(3;s,t,u)\le R_3(s,t,u)-1.$
\end{theorem}

\begin{proof}
Let $N=R_r(k_0,\dots,k_{r-1})-1.$ We will show that any $k$-coloring of $[1,N]$, for some $i \in \{0,1,\dots,k-1\}$, a solution to $\mathcal{L}(t_i)$ with all variables monochromatic in color $i$. Consider any $r$-coloring
$\chi:[1,N]\to\{0,1,\dots,r-1\}$.
We construct an $r$-edge-coloring of the complete graph $K_{N+1}$ on vertex set
$\{0,1,2,\dots,N\}$
by declaring that for $0\le u<v\le N$, the edge $\{u,v\}$ receives the color $\chi(v-u).$
Since $K_{N+1}$ has $N+1=R_r(k_0,\dots,k_{r-1})$ vertices, for $r$-edge-coloring of $K_{N+1}$, it contains a monochromatic clique
$K_{k_i}$ in some color $i$ for $0 \leq i \leq r-1$. Write the vertices of that clique in increasing order as
$a_0<a_1<\cdots<a_{k_i-1}$.
Because every edge of the clique has color $i$, the following $k_i$ positive integers all have color $i$:
\[
x_1=a_1-a_0,\quad x_2=a_2-a_1,\quad \dots,\quad x_{k_i-1}=a_{k_i-1}-a_{k_i-2},\quad x_{k_i}=a_{k_i-1}-a_0.
\]
Moreover,
$x_1+x_2+\cdots+x_{k_i-1}=a_{k_i-1}-a_0=x_{k_i}$.
Thus $x_1,\dots,x_{k_i}$ form a monochromatic color-$i$ solution to $\mathcal{L}(k_i)$.
Since the original coloring of $[1,N]$ was arbitrary, this proves that
$S(r;k_0,\dots,k_{r-1})\le N=R_r(k_0,\dots,k_{r-1})-1$.
\end{proof}

\begin{theorem}{\upshape \cite{HeWigderson2020}}\label{HeWigderson2020}
Fix $r\ge 2$ and fixed integers $k_0,\dots,k_{r-2}\ge 3$.
Then there exists a constant $C_{k_0,\dots,k_{r-2}}>0$ such that for all sufficiently large $u$,
\[
R_r(k_0,\dots,k_{r-2},u)
\le
C_{k_0,\dots,k_{r-2}}\,\frac{u^{\sum_{i=0}^{r-2}(k_i-2)+1}}{(\log u)^{\sum_{i=0}^{r-2}(k_i-2)}}.
\]
\end{theorem}

The next theorem records a general one-large-parameter upper bound obtained by combining
Theorem~\ref{th:ramsey-embedding} and Theorem \ref{HeWigderson2020}.

\begin{corollary}\label{coroll-2}
Fix $r\ge 2$ and fixed integers $k_0,\dots,k_{r-2}\ge 3$.
Then there exists a constant $C_{k_0,\dots,k_{r-2}}>0$ such that for all sufficiently large $u$,
\[
S(r;k_0,\dots,k_{r-2},u)
\le
C_{k_0,\dots,k_{r-2}}\,\frac{u^{\sum_{i=0}^{r-2}(k_i-2)+1}}{(\log u)^{\sum_{i=0}^{r-2}(k_i-2)}}-1.
\]
In particular, fix integers $3\le s\le t$.  Then there exists a constant $C_{s,t}>0$ such that for all sufficiently large $u$,
\[
S(3;s,t,u)\le C_{s,t}\,\frac{u^{s+t-3}}{(\log u)^{s+t-4}}-1.
\]
\end{corollary}

\section{Computer-assisted exact values on finite ranges}

The exact values below were established by means of an exhaustive computer search, as shown in Table \ref{table-3}.

\begin{table}[!htbp]
\footnotesize
\caption{Some exact values of $S(3;s,t,u)$.}
\centering
\begin{tabular}{lc}
\hline
$S(3;s,t,u)$  \\
\hline
\makecell[l]{
$S(3;3,3,8)=59,\ S(3;3,3,9)=68,\ S(3;3,3,10)=77,\ S(3;3,3,11)=86,\ S(3;3,3,12)=94,$\\
$S(3;3,3,13)=104, \ S(3;3,3,14)=113,\ S(3;3,3,15)=122,$\\
$S(3;3,4,8)=67,\ S(3;3,4,9)=78, \ S(3;3,4,10)=86,\ S(3;3,4,11)=98,\ S(3;3,4,12)=106,$ \\
$S(3;3,5,8)=91,\ S(3;3,5,9)=103, \ S(3;3,5,10)=115,$\\
$S(3;4,4,8)=87,\ S(3;4,4,9)=98,\ S(3;4,4,10)=109,$\\
$S(3;4,5,8)=111,$\\
$S(3;4,6,7)=118$.\\
} \\
\hline
\end{tabular}
\label{table-3}
\end{table}

For a fixed integer $N$, we encode the corresponding coloring problem on $[1,N]$ as a Boolean satisfiability instance.
For each integer $i \in [1,N]$ and each color $c \in \{1,2,3\}$, a Boolean variable $v_{i,c}$ is introduced to indicate whether $i$ is assigned color $c$.
The encoding consists of two parts: the exactly-one-color constraints and the constraints forbidding monochromatic solutions of type $L(r)$ in each color class.
The overall decision procedure is presented in Algorithm~\ref{alg:solve_s3_instance},
the clause generation routine for monochromatic configurations is provided in Algorithm~\ref{alg:add_no_mono_constraints},
and the interval search strategy for computing $S(3;s,t,u)$ is outlined in Algorithm~\ref{alg:search_range}; see the Appendix.

\section{Appendix}

\begin{algorithm}[htbp]\scriptsize
\caption{SAT-based decision procedure for a single instance $(N,s,t,u)$}
\label{alg:solve_s3_instance}
\begin{algorithmic}[1]
\Require Positive integers $N,s,t,u$
\Ensure Whether there exists a $3$-coloring of $[1,N]$ such that color $1$ avoids $L(s)$, color $2$ avoids $L(t)$, and color $3$ avoids $L(u)$

\State Initialize a SAT solver $\mathcal{S}$

\For{$i \gets 1$ to $N$}
    \State Introduce Boolean variables $v_{i,1}, v_{i,2}, v_{i,3}$
    \State Add the exactly-one-color constraints for $i$
\EndFor

\State $cnt_s \gets \Call{AddNoMonochromaticLConstraints}{S, s, N, 1}$
\State $cnt_t \gets \Call{AddNoMonochromaticLConstraints}{S, t, N, 2}$
\State $cnt_u \gets \Call{AddNoMonochromaticLConstraints}{S, u, N, 3}$
\State $isSat \gets \Call{Solve}{S}$

\If{$isSat$}
    \State Extract a satisfying assignment from $\mathcal{S}$
    \State Decode it into a coloring $\chi:[1,N]\to\{1,2,3\}$
    \State \Return $(\textsc{SAT}, \chi, (cnt_s,cnt_t,cnt_u))$
\Else
    \State \Return $(\textsc{UNSAT}, \varnothing, (cnt_s,cnt_t,cnt_u))$
\EndIf
\end{algorithmic}
\end{algorithm}

\begin{algorithm}[htbp]\scriptsize
\caption{Generate clauses forbidding monochromatic solutions of $L(r)$ in one color}
\label{alg:add_no_mono_constraints}
\begin{algorithmic}[1]
\Require SAT solver $\mathcal{S}$, integers $r,N$, forbidden color $c$
\Ensure CNF clauses forbidding monochromatic solutions of type $L(r)$ in color $c$

\State $totalClauses \gets 0$

\For{$rhs \gets r-1$ to $N$}
    \ForAll{candidate sets $B \subseteq \{1,2,\dots,rhs-1\}$ generated in increasing order}
        \State $k \gets |B|$
        \State $baseSum \gets \sum_{x\in B} x$
        \State $extraSlots \gets (r-1)-k$
        \State $extraSum \gets rhs-baseSum$

        \If{$extraSlots < 0$ \textbf{or} $extraSum < 0$}
            \State \textbf{continue}
        \EndIf

        \If{$\Call{CanFillExactly}{B, extraSlots, extraSum}$}
            \State $nums \gets B \cup \{rhs\}$
            \State Add the clause
            \[
            \bigvee_{x\in nums}\neg v_{x,c}
            \]
            to $\mathcal{S}$
            \State $totalClauses \gets totalClauses + 1$
        \EndIf
    \EndFor
\EndFor

\State \Return $totalClauses$
\end{algorithmic}
\end{algorithm}

\begin{algorithm}[htbp]\scriptsize
\caption{Range search for $S(3;s,t,u)$}
\label{alg:search_range}
\begin{algorithmic}[1]
\Require Positive integers $s,t,u$, and a search interval $[N_{start},N_{end}]$
\Ensure The first unsatisfiable value in the interval, if it exists

\State $firstUnsat \gets \varnothing$

\For{$N \gets N_{start}$ to $N_{end}$}
    \State $(status,\chi,counts) \gets \Call{SolveS3Instance}{N,s,t,u}$
    \If{$status = \textsc{UNSAT}$}
        \State $firstUnsat \gets N$
        \State \textbf{break}
    \EndIf
\EndFor

\If{$firstUnsat = \varnothing$}
    \State \Return ``No unsatisfiable instance found in $[N_{start},N_{end}]$''
\Else
    \State \Return $firstUnsat$
\EndIf
\end{algorithmic}
\end{algorithm}

\end{document}